\documentclass[12pt,a4paper,reqno]{amsart}
\usepackage[T1]{fontenc}
\usepackage{lmodern}
\usepackage{amsmath,amssymb,mathtools,mathrsfs}
\usepackage{microtype}
\usepackage[a4paper,hmargin=28mm,top=26mm,bottom=28mm,includeheadfoot]{geometry}
\usepackage{tikz-cd}
\usepackage{booktabs}
\usepackage{xurl}
\usepackage{xcolor}
\definecolor{linkblue}{RGB}{35,64,105}
\usepackage[colorlinks=true,linkcolor=linkblue,citecolor=linkblue,urlcolor=linkblue]{hyperref}
\allowdisplaybreaks[1]
\newtheoremstyle{readableplain}{9pt plus 2pt minus 1pt}{8pt plus 2pt minus 1pt}{\itshape}{}{\bfseries}{.}{.5em}{}
\newtheoremstyle{readabledefinition}{9pt plus 2pt minus 1pt}{8pt plus 2pt minus 1pt}{\normalfont}{}{\bfseries}{.}{.5em}{}
\newtheoremstyle{readableremark}{9pt plus 2pt minus 1pt}{8pt plus 2pt minus 1pt}{\normalfont}{}{\itshape}{.}{.5em}{}
\theoremstyle{readableplain}
\newtheorem{theorem}{Theorem}[section]
\newtheorem{proposition}[theorem]{Proposition}
\newtheorem{lemma}[theorem]{Lemma}
\newtheorem{corollary}[theorem]{Corollary}

\newtheorem{conjecture}{Conjecture}[section]
\theoremstyle{readabledefinition}

\theoremstyle{readableremark}
\newtheorem{remark}[theorem]{Remark}
\makeatletter
\renewcommand{\section}{\@startsection{section}{1}{\z@}{-2.5ex\@plus-.8ex\@minus-.2ex}{1.2ex\@plus.2ex}{\normalfont\large\bfseries\raggedright}}
\renewcommand{\subsection}{\@startsection{subsection}{2}{\z@}{-2ex\@plus-.6ex\@minus-.2ex}{.8ex\@plus.2ex}{\normalfont\bfseries\raggedright}}
\renewcommand{\subsubsection}{\@startsection{subsubsection}{3}{\z@}{-1.5ex\@plus-.4ex\@minus-.2ex}{.6ex\@plus.1ex}{\normalfont\bfseries\raggedright}}
\makeatother
\numberwithin{equation}{section}

\DeclareMathOperator{\Spec}{Spec}

\DeclareMathOperator{\Pic}{Pic}

\newcommand{\C}{\mathbb C}

\newcommand{\Pp}{\mathbb P}
\newcommand{\OO}{\mathcal O}

\newcommand{\MM}{\mathcal M}
\newcommand{\Ga}{\mathbb G_a}

\newcommand{\dd}{\mathrm d}

\title[Transverse rational curves and the LeBrun--Salamon Conjecture]{Transverse Rational Curves and the LeBrun--Salamon Conjecture}
\author{Zhengyu Hu}
\address{Zhengyu Hu: Mathematical Sciences Research Center, Chongqing University of Technology, Chongqing 400054, China}
\email{zhengyuhu16@gmail.com}
\author{Kefeng Liu}
\address{Kefeng Liu: Mathematical Sciences Research Center, Chongqing University of Technology, Chongqing 400054, China}
\email{kefengliu@cqut.edu.cn}
\author{Xueyuan Wan}
\address{Xueyuan Wan: Mathematical Science Research Center, Chongqing University of Technology, Chongqing 400054, China}
\email{xwan@cqut.edu.cn}

\makeatletter
\@namedef{subjclassname@2020}{\textup{2020} Mathematics Subject Classification}
\makeatother
\subjclass[2020]{53C26, 14M17, 14J45, 32Q26}
\keywords{positive quaternionic--K\"ahler manifolds, twistor spaces, contact Fano manifolds, rational curves, varieties of minimal rational tangents, additive actions}

\begin{document}
\begin{abstract}
We prove that every compact connected positive quaternionic--K\"ahler manifold of real dimension $4n$ with $n\ge2$ is isometric to a Wolf space after rescaling the metric by a positive constant. This proves the LeBrun--Salamon conjecture. The proof uses transverse rational curves and the contact structure of the twistor space.
\end{abstract}
\maketitle

\section{Introduction}

Positive quaternionic--K\"ahler manifolds are Einstein manifolds with holonomy contained in $\mathrm{Sp}(n)\mathrm{Sp}(1)$. Their twistor spaces carry complex contact structures, which relate their Riemannian geometry to the geometry of contact Fano manifolds. A \emph{Wolf space} is a compact, simply connected, irreducible quaternionic--K\"ahler symmetric space with its canonical invariant Einstein metric, unique up to a positive constant factor. Examples include $\mathbb H\Pp^n$ and $\mathrm{Gr}_2(\C^{n+2})$; see Wolf's classification \cite{Wolf65,Salamon82}. The LeBrun--Salamon conjecture asserts that every compact positive quaternionic--K\"ahler manifold is a Wolf space. It would thus classify these manifolds and their contact Fano twistor spaces.

Salamon developed the differential geometry of the twistor construction, and LeBrun established its relation to contact Fano manifolds and proved the uniqueness results used below \cite{Salamon82,LeBrun95}. The study of contact lines and their tangent directions gives further restrictions on these Fano manifolds; see \cite{Kebekus01,Buczynski10,BKK22}. We use the following formulation of the conjecture.

\begin{conjecture}[LeBrun--Salamon {\cite{LeBrunSalamon94,LeBrun95}}]
\label{conj:LS}
Every compact connected positive quaternionic--K\"ahler manifold of real dimension $4n$, with $n\ge2$, is isometric, after multiplication of its metric by a positive constant, to a Wolf space.
\end{conjecture}

All varieties and stacks are defined over $\C$. We write $\Pp(E)$ for the bundle of one-dimensional subspaces of the fibres of $E$, and $\Ga$ for the additive algebraic group. Each chosen irreducible component of a curve family is given its reduced structure. The algebraic families are constructed before passing to analytic neighbourhoods and holomorphic flows.

Let $(X^{4n},g)$ be such a manifold and let $\pi:Y\to X$ be its twistor space.  Then $Y$ has complex dimension $2n+1$ and fits into the contact sequence
\begin{equation}\label{eq:intro-contact}
0\longrightarrow D\longrightarrow T_Y\xrightarrow{\theta}L\longrightarrow0,
\qquad K_Y^{-1}\simeq L^{\otimes(n+1)},
\end{equation}
where $L$ is ample and $Y$ admits a K\"ahler--Einstein metric.  Each twistor fibre $C_x=\pi^{-1}(x)$ is a smooth embedded rational curve with
\begin{equation}\label{eq:intro-fibre}
C_x\simeq\Pp^1,\qquad N_{C_x/Y}\simeq\OO_{\Pp^1}(1)^{\oplus 2n},\qquad
\theta|_{T C_x}:T C_x\xrightarrow{\sim}L|_{C_x}.
\end{equation}
These properties are part of the standard twistor construction \cite[Theorem~2.1]{LeBrun95}.  The twistor space also has a fixed-point-free antiholomorphic involution $\sigma:Y\to Y$ that preserves each fibre and acts on it as the antipodal map.

The contact Fano reduction of \cite[Proposition~2.3]{BKK22} (see also \cite{KPSW00}) gives the following three alternatives:
\begin{equation}\label{eq:intro-reduction}
Y\simeq\Pp^{2n+1},\qquad\text{or}\qquad
Y\simeq\Pp(T^*\Pp^{n+1}),\qquad\text{or}\qquad
\Pic(Y)=\mathbb Z[L].
\end{equation}
In the first case, $Y=\Pp^{2n+1}$ is the twistor space of $\mathbb H\Pp^n$; in the second, $Y=\Pp(T^*\Pp^{n+1})$ is the twistor space of $\mathrm{Gr}_2(\C^{n+2})$. LeBrun's uniqueness theorem \cite[Corollaries~3.3--3.4]{LeBrun95} identifies the original metric with the corresponding Wolf-space metric up to a positive constant factor. It remains to treat the third case.

We prove the following theorem.

\begin{theorem}\label{thm:main}
Let $(X^{4n},g)$ be a compact connected positive quaternionic--K\"ahler manifold with $n\ge2$.  Then, after multiplying $g$ by a positive constant, $(X,g)$ is isometric to a Wolf space.
\end{theorem}

The proof uses transverse rational curves $C\subset Y$ of $L$-degree two, that is,
\(\deg(L|_C)=2\), where $L$ is the contact line bundle. A smooth rational
curve $C\subset Y$ is transverse if $\theta|_{TC}$ has no zero, and balanced if
$N_{C/Y}\simeq\OO(1)^{\oplus2n}$. We construct an intrinsic two-point tensor on
the universal family and extend its contraction with the contact differential
across the nodal boundary. The resulting vector field integrates to an
algebraic additive action. First-jet interpolation shows that the induced local
maps preserve minimal rational tangents. A theorem of Hwang and Mok then gives
global contact vector fields whose Hamiltonians generate the contact line
bundle. Beauville's theorem identifies the twistor space, and LeBrun's
uniqueness theorem identifies the metric.

\begin{remark}
This project was developed over several months, and its main result and core argument were obtained well before the appearance of the OpenAI manuscript \cite{OpenAI2026}. Our proof was developed independently and follows a different approach. The OpenAI manuscript smooths a chosen pair of contact lines into very free conics and then uses a diagonal-jet construction. Our argument starts from transverse stable domains of contact degree two, extends an intrinsic tensor across the nodal boundary, and combines the resulting additive action, generated by a nilpotent endomorphism, with first-jet interpolation. This gives a direct incidence-based proof without a separate smoothing or diagonal-jet argument.
\end{remark}

\section{Stable domains of contact degree two}\label{sec:stable}

We retain the notation and assumptions of the Introduction.  In particular, $Y$ is a smooth projective contact manifold of dimension $2n+1$, with contact sequence
\[
0\longrightarrow D\longrightarrow T_Y\xrightarrow{\theta}L\longrightarrow0,
\]
where $L$ is ample.  Fix a smooth embedded balanced transverse rational curve
$C_0\subset Y$; such curves are provided by the twistor fibres in
\eqref{eq:intro-fibre}.  Set
\[
\beta=[C_0]\in H_2(Y,\mathbb Z).
\]
Transversality gives $T_{C_0}\simeq L|_{C_0}$, and hence
$L\cdot\beta=2$.  For a stable map $h:C\to Y$, its $L$-degree is
$\deg(h^*L)$, namely the sum of the degrees on the irreducible components of
$C$.  Thus every stable map representing $\beta$ has $L$-degree two.

Let $\overline{M}_{0,0}(Y,\beta)$ denote the proper Deligne--Mumford stack of
unmarked genus-zero stable maps of class $\beta$ to $Y$.  Let $\MM$ be the
reduced irreducible component containing the point represented by the inclusion
$C_0\hookrightarrow Y$.  Its coarse moduli space is projective by the standard projective construction of stable-map moduli spaces \cite{FultonPandharipande97,BehrendManin96}.  Choose an \emph{\'etale atlas} $S\to\MM$, that is,
a surjective \'etale morphism from a scheme (see \cite[Tag~06N3]{StacksProject}).  Write
\[
p:\mathcal C\longrightarrow S,\qquad f:\mathcal C\longrightarrow Y
\]
for the universal curve and universal map pulled back to $S$.  The base $S$ is
reduced, and the smooth-domain locus is dense.  We carry out the constructions
on the stack and then descend them to the coarse spaces.

\begin{lemma}\label{lem:fibre}
Let $C=\mathcal C_s$ be a geometric fibre of $p$, and write $f_s=f|_C$.  Then $C$ is either a smooth projective line or the union of two smooth projective lines meeting at a single node.  In the second case, $f_s^*L$ has degree one on each component.  No component of $C$ is mapped to a point by $f_s$.
\end{lemma}
\begin{proof}
The curve $C$ is connected and nodal, with arithmetic genus zero.  Hence each irreducible component $C_i$ is a smooth projective line.  Its dual graph $\Gamma$, with one vertex for each component and one edge for each node, is a tree: it is connected and has no cycles.
Set $d_i=\deg((f_s|_{C_i})^*L)$.  If $f_s$ maps $C_i$ to a point, then $d_i=0$; such a component is called \emph{contracted}.  Otherwise $d_i$ is a positive integer, because $L$ is an ample line bundle.  Degrees add over the components, so
\[
\sum_i d_i=\deg(f_s^*L)=L\cdot\beta=2.
\]
Thus there are at most two noncontracted components.  To exclude contracted components, we use stability.

In an unmarked stable map, each contracted rational component meets the other components in at least three points; see \cite[Section~1.1]{FultonPandharipande97}.  Otherwise infinitely many automorphisms of the component would fix its attachment points and preserve the map.  Thus a vertex of $\Gamma$ corresponding to a contracted component has at least three incident edges.  In particular, each leaf, that is, a vertex with only one incident edge, corresponds to a noncontracted component.

If a contracted component existed, removing its vertex from $\Gamma$ would leave at least three disjoint branches.  In each branch, choose a vertex farthest from the removed vertex.  Since $\Gamma$ is a finite tree, this vertex is a leaf of the original graph.  This gives at least three distinct leaves, hence at least three noncontracted components, contrary to the degree bound.  Thus no component is contracted.

It follows that $C$ has either one or two components.  In the first case, $C\simeq\Pp^1$.  In the second, the tree has two vertices joined by one edge, so the two components meet at exactly one node.  Both degrees are one, since they are positive integers with sum two.
\end{proof}

Put $F=f^*L$ and $\omega=\omega_{\mathcal C/S}$.  The relative dualizing sheaf $\omega_{\mathcal C/S}$ of $p:\mathcal C\to S$ plays the role of the canonical bundle in relative Serre duality.  It is invertible because $p$ is a family of nodal curves.  On a smooth fibre $C$ it restricts to the canonical bundle $\omega_C$, while if $C=C_1\cup C_2$ has a node $q$, then
\[
\omega_{\mathcal C/S}|_{C_i}\simeq\omega_{C_i}(q_i),
\]
where $q_i\in C_i$ is the point corresponding to the node.

\begin{lemma}\label{lem:linebundle}
There is an invertible sheaf
\[
B=p_*(F\otimes\omega)
\]
and a canonical evaluation isomorphism
\begin{equation}\label{eq:eval}
p^*B\simeq F\otimes\omega,
\qquad F\simeq\omega^{-1}\otimes p^*B.
\end{equation}
The relative contact differential determines a section $s\in H^0(S,B)$.  Near
a node, choose an \'etale-local equation $xy=t$ over $U=\Spec A$, where
$t\in A$.  Here $t$ is the actual element appearing in the node equation,
and it is the parameter used in the divisibility statement.  Then $s|_U$ is
divisible by $t$ as a section of $B|_U$;
equivalently,
\[
s|_U\in t\,H^0(U,B|_U).
\]
\end{lemma}
\begin{proof}
Set $G=F\otimes\omega$.  We compute the restriction of $G$ to each fibre of $p$.  On a smooth fibre $C\simeq\Pp^1$, we have
$
F|_C\simeq\OO_{\Pp^1}(2),
$ $
\omega|_C\simeq\omega_C\simeq\OO_{\Pp^1}(-2),
$
so $G|_C\simeq\OO_{\Pp^1}$.  If the fibre is $C=C_1\cup C_2$, then each component has $F$-degree one, while
$
\omega|_{C_i}\simeq\omega_{C_i}(\text{node})\simeq\OO_{\Pp^1}(-1).
$
Thus $G$ has degree zero on both components.  Since the dual graph is a tree, a line bundle of multidegree zero on $C$ is trivial.  In either case,
$
H^0(C,G|_C)\simeq\C,
$ $H^1(C,G|_C)=0.$
Cohomology and base change apply as follows.  Since $p$ is proper, flat and
of finite presentation and $G$ is invertible, $G$ is flat over $S$.  The
preceding calculation gives $h^0(C_b,G|_{C_b})=1$ and
$h^1(C_b,G|_{C_b})=0$ for every geometric point $b$ of $S$.  Hence
cohomology and base change (see \cite[Lemmas~36.30.4 and 36.32.3, Tags~0B91 and 0B9S]{StacksProject}) show that $B=p_*G$ is a line bundle and that
$B\otimes k(b)\simeq H^0(C_b,G|_{C_b})$.  Under these identifications,
the evaluation map
\[
p^*B\longrightarrow G
\]
restricts on each fibre to the usual evaluation map
$H^0(C_b,G|_{C_b})\otimes\OO_{C_b}\to G|_{C_b}$, which is an isomorphism
because $G|_{C_b}\simeq\OO_{C_b}$.  Thus the evaluation map is an isomorphism
of line bundles.  This proves the first isomorphism in \eqref{eq:eval}; the
second follows by tensoring with $\omega^{-1}$.  The same calculation applied to $\OO_{\mathcal C}$ gives
$
p_*\OO_{\mathcal C}=\OO_S,
$
because every fibre is connected and has only constant global functions.

Pulling back the contact morphism $\theta:T_Y\to L$ by $f$ and composing it
with the differential of $f$ gives a section of
$\Omega^1_{\mathcal C/S}\otimes F$.  Its image under the natural morphism
induced by the canonical comparison map
\[
\rho:\Omega^1_{\mathcal C/S}\longrightarrow\omega
\]
between relative differentials and the relative dualizing sheaf (see
\cite[Chapter~X, \S~2]{ACG11}),
\[
\rho\otimes\mathrm{id}_F:\Omega^1_{\mathcal C/S}\otimes F
\longrightarrow\omega\otimes F\simeq p^*B,
\]
is the pullback $p^*s$ of a unique section $s\in H^0(S,B)$.  The uniqueness
follows from $p_*\OO_{\mathcal C}=\OO_S$.

It remains to prove divisibility.  Work \emph{\'etale-locally} near a node,
using the standard local structure of nodal families (see
\cite[Lemma~53.21.1]{StacksProject}).  Let $A$ be the local ring at the
corresponding point of the chosen \'etale base.  Choose $t\in A$ and the
model
\[
R=A[x,y]/(xy-t).
\]
Let $Q$ be the local ring at a point over the node in a common \'etale
neighbourhood of the curve and $\operatorname{Spec}R$.  Choose the standard
local generator $\eta$ of $\omega_{R/A}$ associated with the equation
$xy=t$.  Then
\[
\rho(dx)=x\eta,\qquad \rho(dy)=-y\eta,
\]
so $\operatorname{im}\rho=(x,y)\omega_{R/A}$.  The same image formula holds
after pullback to $Q$.

Trivialize $B$ and $F$ near the base point and write $s=a e$ with $a\in A$.
Since $p^*s$ lies in the image of $\rho\otimes\mathrm{id}_F$, the image of
$a$ in $Q$ belongs to $(x,y)Q$.  Now
\[
R/(x,y)\simeq A/(t),
\]
and the induced map
\[
A/(t)\longrightarrow Q/(x,y)Q
\]
is a flat local homomorphism, hence faithfully flat.  Since $a$ maps to zero
in $Q/(x,y)Q$, it is zero in $A/(t)$.  Thus $a\in tA$.  Since the
smooth-domain locus is dense in every irreducible component of the reduced
base, $t$ avoids all minimal primes of $A$; after shrinking the chosen
\'etale base, it is a non-zero-divisor.  The quotient $a/t$ is therefore
unique in $A$, and the same is true in every further \'etale refinement.
Thus $s$ is divisible by $t$ as a section of $B$.  The quotient itself
depends on the choice of the local smoothing equation (for example, replacing
$t$ by $ut$ multiplies it by the inverse unit $u^{-1}$); what is invariant is
the divisibility statement and, in Proposition~\ref{prop:extend}, the product
of this quotient with the corresponding regular tensor $Q$.  This proves the
assertion.
\end{proof}

\begin{lemma}\label{lem:liebundle}
The dual of $\Omega^1_{\mathcal C/S}\to\omega$ gives an injective map
\begin{equation}\label{eq:injection}
\iota:\omega^{-1}\hookrightarrow T_{\mathcal C/S}
\end{equation}
whose image is closed under the relative Lie bracket.  The sheaf $E=p_*\omega^{-1}$ is locally free of rank three and carries the induced $\OO_S$-bilinear Lie bracket.
\end{lemma}
\begin{proof}
On the smooth locus of $p$, the map \eqref{eq:injection} is the usual identification of the inverse canonical line with the tangent line.  At a node, use the local model $R=A[x,y]/(xy-t)$.  The formulas $dx\mapsto x\eta$ and $dy\mapsto-y\eta$ show that the dual map sends
\[
\eta^{-1}\longmapsto D_0=x\partial_x-y\partial_y.
\]
If $uD_0=0$, applying this derivation to $x$ and $y$ gives $ux=uy=0$.  Using the relation $xy=t$, every element $u\in R$ has a unique expression
\[
u=a_0+\sum_{i\ge1}a_ix^i+\sum_{j\ge1}b_jy^j,
\]
with coefficients in $A$ and only finitely many nonzero terms.  Comparing the coefficients of the positive powers of $x$ in $ux=0$ gives $a_0=a_i=0$.  Then $uy=0$ gives $b_j=0$ for every $j$.  Thus $u=0$, proving injectivity on the local model; injectivity on an \'etale chart follows by flatness.  The calculation holds over any base ring $A$, so it remains valid after base change.

The image is closed under the Lie bracket: its local sections have the form $uD_0$, and
\[
[uD_0,vD_0]=(uD_0(v)-vD_0(u))D_0,
\]
which is again a multiple of $D_0$.

To compute the cohomology of $\omega^{-1}$, consider its restriction to each fibre.  On a smooth fibre it is $\OO_{\Pp^1}(2)$, so $h^0=3$ and $h^1=0$.  On a reducible fibre $C_1\cup C_2$, it restricts to $\OO_{\Pp^1}(1)$ on each component.  A global section is a pair of sections whose values agree in the common fibre at the node.  Each component contributes two parameters, and matching the values imposes one independent linear condition.  Thus $h^0=2+2-1=3$; the normalization exact sequence also gives $h^1=0$.  Cohomology and base change show that $E=p_*\omega^{-1}$ is locally free of rank three.

The bracket on $\omega^{-1}$ induces a bracket on $E$.  Since relative vector fields annihilate functions pulled back from $S$, we have $[aV,bW]=ab[V,W]$ for local functions $a,b$ on $S$ and local sections $V,W$ of $E$.  Hence this bracket is $\OO_S$-bilinear.
\end{proof}

\begin{remark}
The map \eqref{eq:injection} need not be surjective at a node.  On the nodal fibre $xy=0$, a local derivation has the form
$
x\,a(x)\partial_x+y\,b(y)\partial_y.
$
It belongs to the image of $\iota$ precisely when $a(0)+b(0)=0$. For example, $x\partial_x$ is a derivation but does not lie in the image.  We use only the image of $\iota$, which is closed under the Lie bracket.
\end{remark}

\begin{lemma}\label{lem:conic}
The evaluation map $p^*E\to\omega^{-1}$ defines the relative anticanonical morphism
\[
\Phi:\mathcal C\longrightarrow\Pp(E^*).
\]
This morphism is a closed immersion, and its image is a plane conic in every fibre.  Near a nodal fibre, an \'etale-local equation for this conic is
\begin{equation}\label{eq:conic}
XY=tZ^2.
\end{equation}
On $Z\ne0$ the three sections corresponding to $X,Y,Z$ are
\begin{equation}\label{eq:frame}
e_x=xD_0,\qquad e_y=yD_0,\qquad e_0=D_0,
\end{equation}
and they extend to a frame of $E$.
\end{lemma}
\begin{proof}
The line bundle $\omega^{-1}$ has degree $2$ on a smooth fibre and degree $1$ on each component of a reducible fibre.  Thus, on a smooth fibre $C\simeq\Pp^1$, it is $\OO_{\Pp^1}(2)$, and the complete linear system gives the quadratic Veronese embedding
$
\Pp^1\hookrightarrow\Pp^2.
$
Its image is a smooth plane conic.  On a reducible fibre $C=C_1\cup C_2$, the restriction to each component is $\OO_{\Pp^1}(1)$.  A section is a pair of linear sections on $C_1$ and $C_2$ with the same value at the node.  Thus each component maps isomorphically to a line, and the two lines meet at the image of the node.  The image is a nodal plane conic.

Since $H^1(C,\omega^{-1}|_C)=0$ on every fibre, cohomology and base change show that the evaluation map $p^*E\to\omega^{-1}$ is surjective and that it gives the complete linear system on every fibre.  These linear systems are very ample on the smooth fibres and separate the two components and the node on a reducible fibre.  Thus $\omega^{-1}$ is relatively very ample.  Hence $\Phi$ is a closed immersion.

Near a nodal fibre, choose coordinates in the target plane so that the two components of the central conic are $X=0$ and $Y=0$.  Flatness and a rescaling of the smoothing parameter then put the family in the standard form
\[
XY=tZ^2.
\]
On the chart $Z\ne0$, put $x=X/Z$ and $y=Y/Z$; the equation becomes $xy=t$.  The local generator $D_0=x\partial_x-y\partial_y$ of the image of $\omega^{-1}$ gives the sections
\[
e_x=xD_0,\qquad e_y=yD_0,\qquad e_0=D_0.
\]
They are regular on the two standard charts of the nodal curve.  When $t\ne0$, they correspond to the three standard quadratic sections $x$, $1$, and $y=t/x$ of $\OO(2)$.  When $t=0$, they correspond to the pairs of linear sections $(x,0)$, $(1,1)$, and $(0,y)$ on the two components; these pairs agree at the node and form a basis of the three-dimensional space of global sections.  They are therefore linearly independent on every fibre, and Nakayama's lemma shows that they form a local frame of $E$.
\end{proof}

\section{A regular additive action on the incidence source}\label{sec:action}

\subsection{The intrinsic tensor and its extension}
Let $T=T_{\Pp^1}$.  If $E$ and $F$ are vector bundles on a variety $X$, their external tensor product is
$
E\boxtimes F:=\operatorname{pr}_1^*E\otimes\operatorname{pr}_2^*F
$
on $X\times X$, where $\operatorname{pr}_1$ and $\operatorname{pr}_2$ are the two projections.  Thus
$
T\boxtimes T=\operatorname{pr}_1^*T\otimes\operatorname{pr}_2^*T,
$
whose fibre at an ordered pair $(a,b)$ is
$
T_a\Pp^1\otimes T_b\Pp^1.
$
The symbol $\boxtimes$ indicates that the factors are pulled back from the two copies of the curve.  In contrast, $T\otimes T$ is a bundle on the curve itself; the tensor below is defined on pairs $(a,b)$, whereas $T\otimes T$ on $\Pp^1$ has fibre $T_a\Pp^1\otimes T_a\Pp^1$ at a single point $a$.

On a smooth projective line, the formula
\begin{equation}\label{eq:q}
q(a,b)=(a-b)^2\frac{\partial}{\partial a}\otimes\frac{\partial}{\partial b}
\end{equation}
defines a global section of $T\boxtimes T$ on $\Pp^1\times\Pp^1$.  It is invariant under fractional linear changes of coordinate, and hence defines the same tensor on every smooth fibre.  Along the diagonal $\Delta=\{a=b\}$, the bundle $T\boxtimes T$ restricts canonically to $T^{\otimes2}$, and the factor $(a-b)^2$ shows that $q$ vanishes to order two along $\Delta$.  After contraction with the contact differential, this is the double zero of the induced vector field at the fixed endpoint.

\begin{lemma}\label{lem:q-pole}
On the smooth-domain locus of $\mathcal C\times_S\mathcal C$, the tensor $q$ has, in the conic model \eqref{eq:conic}, the expression
\begin{equation}\label{eq:qconic}
q(a,b)=\frac{x_ay_b+y_ax_b-2t}{t}\,D_{0,a}\otimes D_{0,b}.
\end{equation}
Consequently, the tensor $tq$ extends regularly across the locus $t=0$, where $t$ is the actual element appearing in the local nodal equation $xy=t$.
\end{lemma}
\begin{proof}
For $t\ne0$ put $y=t/x$.  Then $D_0=x\partial_x$, and the right-hand side of \eqref{eq:qconic} becomes
\[
\left(\frac{x_a}{x_b}+\frac{x_b}{x_a}-2\right)x_ax_b\partial_{x_a}\otimes\partial_{x_b}
=(x_a-x_b)^2\partial_{x_a}\otimes\partial_{x_b}.
\]
The numerator is regular in the frame \eqref{eq:frame}.  The only possible pole is therefore the factor $t^{-1}$.
\end{proof}

\begin{remark}
There is also an intrinsic description of this calculation.  In the frame \eqref{eq:frame}, the bracket satisfies
\[
[e_0,e_x]=e_x,\quad [e_0,e_y]=-e_y,\quad [e_x,e_y]=-2t e_0.
\]
The Killing form has matrix $$\left(\begin{matrix}0&-4t&0\\-4t&0&0\\0&0&2\end{matrix}\right),$$ and for $t\ne0$ the tensor $-4K^{-1}\in E\otimes E$ evaluates to $q$.  Thus the pole does not depend on the chosen frame.
\end{remark}

Let $\rho_a,\rho_b:\mathcal C\times_S\mathcal C\to\mathcal C$ be the projections.  Contracting the first factor of $q$ with $\theta\circ df$ gives, on the smooth-domain locus,
\[
\widetilde V=(\theta\circ df)_a q\in H^0\bigl(\rho_a^*F\otimes\rho_b^*\omega^{-1}\bigr).
\]

\begin{proposition}\label{prop:extend}
The section $\widetilde V$ extends regularly over $\mathcal C\times_S\mathcal C$.
\end{proposition}
\begin{proof}
Under the identification $F\otimes\omega\simeq p^*B$ from Lemma~\ref{lem:linebundle}, the relative contact differential is represented by the section $s$.  By Lemma~\ref{lem:q-pole}, $tq$ is regular.  Locally, the section $\widetilde V$ is given by
$
\widetilde V=\frac{s}{t}\,Q,
$
where $Q$ is a regular tensor.

Lemma~\ref{lem:linebundle} also shows that the coefficient of $s$ is divisible by $t$.  In a local frame of $B$, we may therefore write $s=t\,s_1$ with $s_1$ regular.  Hence
$
\widetilde V=s_1Q
$
is regular across the nodal fibre.

To check compatibility, take a common \'etale refinement of two such nodal
charts.  The refinement is flat over the reduced base, and the two-point
source is flat over that base; hence every irreducible component of the
refinement meets the smooth-domain locus.  The two local expressions agree
there, because both restrict to the canonical section $\widetilde V$ on the
smooth locus.  The two-point source is reduced: it is flat and of finite
presentation over the reduced base, and its geometric fibres are products of
reduced nodal curves, hence geometrically reduced.  Since the target is
locally free, equality on the dense smooth locus gives equality on the whole
refinement.  Faithfully flat descent then shows that the local sections glue
to a global regular section on
$\mathcal C\times_S\mathcal C$.
\end{proof}

\subsection{Fixing the first endpoint}
Fix a balanced transverse curve $C_0\subset Y$ and a point $z\in Y$ lying on $C_0$.  Let $f_0:C_0\hookrightarrow Y$ denote the inclusion map, and let $a_0\in C_0$ be the source point satisfying
$
 f_0(a_0)=z.
$
Let $\mathcal U\to\MM$ be the universal domain over the stable-map component $\MM$ introduced in Section~\ref{sec:stable}, and let
$\operatorname{ev}:\mathcal U\to Y$ be its universal evaluation map.  Define $I_z$ to be the reduced closure in $\mathcal U$ of the irreducible component of $\operatorname{ev}^{-1}(z)$ containing $(f_0,a_0)$.  Thus a point of $I_z$ is a stable map together with a first source point $a$ whose image is $z$.

We use one and the same ordered two-point source throughout this section:
\[
 \mathcal C_{I_z}:=\mathcal U\times_{\MM} I_z
 \longrightarrow I_z.
\]
A point of $\mathcal C_{I_z}$ is therefore a triple $(f,a,b)$, where $f$ is a stable map, $a$ and $b$ are (possibly equal) points on its source curve, and $f(a)=z$.  In particular, $b$ is allowed to be a node or to coincide with $a$; no extra contracted component is inserted to stabilize the two points.  This is the source on which the constructions below are made.  It is deliberately the relative curve just defined, rather than the two-marked stable-map compactification: the latter inserts a contracted component when the two marks collide or meet a node.  On an etale atlas of $I_z$, this agrees with the pullback notation used in the preceding local calculations.

The image of the first point is fixed, so the factor $\rho_a^*F$ is the fixed one-dimensional vector space $L_z$.  Choose $\ell\in L_z\setminus\{0\}$ and a linear functional $\ell^*:L_z\to\C$ with $\ell^*(\ell)=1$.  Contraction of the extended section with $\ell^*$ in the first factor gives
\begin{equation}\label{eq:V}
V=(\ell^*\otimes1)\widetilde V\in H^0(\mathcal C_{I_z},\omega^{-1}_{\mathcal C_{I_z}/I_z}).
\end{equation}

\begin{proposition}\label{prop:Ga}
Under the inclusion $\omega^{-1}\hookrightarrow T_{\mathcal C_{I_z}/I_z}$, the section $V$ is a relative derivation.  It integrates to an algebraic $\Ga$-action on $\mathcal C_{I_z}$ over $I_z$.  The action fixes the first-point section.  On a smooth transverse domain, its infinitesimal generator has a double zero at the first point and is nonzero elsewhere.
\end{proposition}
\begin{proof}
Let $E_{I_z}=p_*\omega^{-1}$, a vector bundle of rank three by Lemma~\ref{lem:liebundle}.  The image of $\omega^{-1}$ is closed under Lie brackets, so $V$ defines the endomorphism
\[
N=\operatorname{ad}(V):E_{I_z}\longrightarrow E_{I_z},
\qquad N(W)=[V,W].
\]

First consider the smooth-domain locus.  Choose a coordinate $u$ with the first point at $u=0$.  The double zero of $V$ then gives
\[
V=c\,u^2\partial_u,
\]
where $c$ is a regular local coefficient, and $c\ne0$ at points
corresponding to smooth transverse domains.

In the basis
$
\partial_u,$ $u\partial_u,$ $u^2\partial_u
$
of the three-dimensional space of vector fields on $\Pp^1$, the operator $N=\operatorname{ad}(V)$ satisfies
$
N^3=0.
$

We next verify this identity on a nodal fibre.  In the local model $xy=t$, write the first point as $(A,B)$, so that $AB=t$.  The expression for $V$ is
$
V=h(Ay+Bx-2t)D_0
$
for a regular scalar $h$.  A direct calculation in the frame \eqref{eq:frame} gives $N=hM$, where
\[
M=\begin{pmatrix}-2t&0&-B\\0&2t&A\\2At&-2Bt&0\end{pmatrix},
\qquad M^3=0.
\]
Thus $N^3=0$ in the nodal model as well.  Since the base is reduced and every irreducible component contains a dense smooth-domain locus,  $N^3=0$ holds on all of $I_z$.

The relative conic is defined in $\Pp(E_{I_z}^*)$ by the quadratic relation
$
e_xe_y-te_0^2=0.
$
The endomorphism $N$ acts as a derivation on the symmetric algebra $\operatorname{Sym}(E_{I_z})$.  The matrix above shows that this derivation annihilates the quadratic relation.  The resulting linear transformations therefore preserve the relative conic.

Since $N^3=0$, the exponential is the polynomial
\[
G_\tau=1+\tau N+\frac{\tau^2}{2}N^2.
\]
The nilpotence relation gives
$
G_{\tau+\upsilon}=G_\tau G_\upsilon,
$
so $G_\tau$ defines an algebraic action of the additive group $\Ga$ on the relative conic.  Since $\Pp(E_{I_z}^*)$ parametrizes lines in the dual bundle, this action is given by
$
[\lambda]\longmapsto[\lambda\circ G_\tau].
$

At the first point, the evaluation functional is killed by $N$.  Hence its line in $E_{I_z}^*$ is unchanged by $G_\tau$, and the first-point section is fixed.  On a smooth transverse fibre, $V=c u^2\partial_u$ with $c\ne0$.  This field has a double zero at $u=0$ and no other zero.
\end{proof}

\subsection{The projective source and evaluation}
Let $U_z$ be the reduced irreducible component of the coarse space of $\mathcal C_{I_z}$ that contains the smooth locus of triples $(C,z,w)$ with $C$ balanced transverse and $w\ne z$.  Let
\[
\nu:W_z\longrightarrow U_z
\]
be its normalization.  Thus $W_z$ is normal and $\nu$ is finite and birational.

\begin{lemma}\label{lem:projective}
The variety $W_z$ is projective.  Evaluation at the second point descends to a morphism
\begin{equation}\label{eq:evalmap}
E_z:W_z\longrightarrow Y.
\end{equation}
\end{lemma}
\begin{proof}
The universal domain $\mathcal U\to\MM$ is projective.  Indeed, for
$m\ge2$ the line bundle
$\omega_{\mathcal U/\MM}\otimes \operatorname{ev}^*L^{\otimes m}$ has positive
degree on every component of every stable map: its degree is $-2+2m$ on a
smooth degree-two fibre, $-1+m$ on a degree-one component of a reducible
fibre, and at least $1$ on a contracted component by stability.  It is
therefore relatively ample; this is also the standard projective presentation
of the universal stable-map curve, cf.\ \cite{FultonPandharipande97,BehrendManin96}.
The component $\MM$ has projective coarse space, and $I_z\subset\mathcal U$ is
a closed substack.  After taking a positive tensor power that descends to the
coarse space, the two-point source
$\mathcal C_{I_z}=\mathcal U\times_{\MM}I_z$ has a projective coarse space,
and so does its closed irreducible component $U_z$.

The normalization of the projective variety $U_z$ is finite, hence $W_z$ is projective.  The evaluation at the second point is the restriction of the universal map from the factor $\mathcal U$ in $\mathcal C_{I_z}$.  It is invariant under isomorphisms of stable maps and therefore descends to $U_z$; composing with the normalization map gives the morphism $E_z$.
\end{proof}

\begin{lemma}\label{lem:action-normalization}
The action in Proposition~\ref{prop:Ga} descends to the coarse space and lifts uniquely to an algebraic action on $W_z$.
\end{lemma}
\begin{proof}
The construction in Proposition~\ref{prop:Ga} commutes with isomorphisms of stable maps.  In characteristic zero, the relevant Deligne--Mumford stacks are tame.  We use the standard base-change property for tame coarse moduli spaces (in particular, the coarse space of $\Ga\times\mathcal C_{I_z}$ is $\Ga\times$ the coarse space of $\mathcal C_{I_z}$), together with the coarse-space theorem of Keel--Mori \cite{KeelMori97,AOV08,StacksProject}; hence the action descends to the coarse space of $\mathcal C_{I_z}$.  The component $U_z$ is invariant: it contains a dense smooth incidence locus, and the connected group $\Ga$ cannot permute this component with a different irreducible component.

For each $\tau\in\Ga$, the descended map sends the generic point of $U_z$ to the generic point, because $G_\tau$ is invertible.  Thus the induced map
$
\Ga\times W_z\to U_z
$
is dominant.  Since $\Ga\times W_z$ is normal, the universal property of normalization gives a unique lift
\[
\Ga\times W_z\longrightarrow W_z.
\]
Uniqueness of the lift implies the group law, so the lift defines an algebraic $\Ga$-action on $W_z$.
\end{proof}

\begin{lemma}\label{lem:etale}
If $C$ is a balanced transverse curve in the chosen local incidence family and $w\in C\setminus\{z\}$, then $W_z$ is smooth at $(C,z,w)$ and $E_z$ is \'etale there.  In particular, $E_z$ is surjective and generically finite.
\end{lemma}
\begin{proof}
The normal bundle of $C$ is
$
N_{C/Y}\simeq\OO_{\Pp^1}(1)^{\oplus 2n}.
$
Since $H^1(C,N_{C/Y})=0$, the deformation space of $C$ in $Y$ is smooth at $C$; see \cite[Chapter~II]{Kollar96}.  The global generation of $N_{C/Y}$ implies that evaluation at a marked point is a submersion.  The infinitesimal deformations that keep the curve through $z$ are therefore
$
H^0(C,N_{C/Y}(-z)).
$
Motion of the second point along $C$ contributes the tangent space $T_wC$.

The differential of the second evaluation map restricts to the inclusion on $T_wC$.  Modulo this tangent line, it is the evaluation map
\[
H^0(C,N_{C/Y}(-z))\longrightarrow (N_{C/Y})_w,
\]
which is an isomorphism when $w\ne z$.  Indeed, for each summand $\OO(1)$, a section vanishing at $z$ is a constant multiple of the unique section of $\OO(1)(-z)\simeq\OO$, and its value at any other point $w$ is nonzero.  The assertion follows by taking $2n$ copies.

The tangent space of the source has dimension
\[
\dim H^0(C,N_{C/Y}(-z))+\dim T_wC=2n+1,
\]
which is the dimension of $Y$.  The differential of $E_z$ is therefore an isomorphism.  The smooth incidence chart near $(C,z,w)$ identifies with an open subset of $W_z$.  Thus $W_z$ is smooth there, and the isomorphism of differentials shows that $E_z$ is \'etale.

Since $W_z$ is projective, $E_z$ is proper and its image is closed.  Being \'etale at one point, it also contains an open neighbourhood of $E_z(C,z,w)$ in $Y$.  This nonempty open subset is dense because $Y$ is irreducible, and the image is closed; hence the image is all of $Y$.  Finally, $\dim W_z=\dim Y=2n+1$ and $E_z$ is dominant, so its general fibre has dimension zero; equivalently, $E_z$ is generically finite.
\end{proof}

\section{The local contact vector field}\label{sec:contactfield}

\begin{proposition}\label{prop:contact}
On every \'etale chart of $E_z$ contained in the balanced transverse locus, the transported field $(E_z)_*V$ preserves the contact distribution $D$.  Its contact Hamiltonian is nonzero at every point of the chosen curve different from $z$.
\end{proposition}
\begin{proof}
Choose a small analytic neighbourhood $B$ in the family of embedded balanced transverse curves through $z$.  Trivialize the universal curve as $B\times\Pp^1$ and write the evaluation map as
\[
e:B\times\Pp^1\longrightarrow Y.
\]
The first point is the section $u=0$.  Choose the coordinate $u$ so that, for every $b\in B$, the contact form on the tangent direction of the curve at $u=0$ takes the fixed nonzero value $\ell\in L_z$.  With this trivialization of the pulled-back contact line bundle, the contact form is
\[
\beta=\left(du+\sum_i A_i(b,u)\,db_i\right)\otimes\partial_u.
\]
For each $i$, the coefficient $A_i(b,\cdot)\partial_u$ is a global vector field on $\Pp^1$.  The first section is fixed at $z$, so
$
A_i(b,0)=0.
$
The linear term also vanishes.  To check this, choose a local frame of $L$ near $z$ and write
\[
N_i=\theta\bigl(de(\partial_{b_i})\bigr),
\qquad
\Delta=\theta\bigl(de(\partial_u)\bigr),
\qquad
A_i=\frac{N_i}{\Delta}.
\]
At $u=0$, the identity $e(b,0)=z$ gives
$
de(\partial_{b_i})|_{(b,0)}=0.
$
Thus $N_i(b,0)=0$.  Upon differentiation in the $u$-direction, the term involving $\mathrm d\theta$ vanishes because one of its arguments is $de(\partial_{b_i})|_{(b,0)}=0$.  The remaining term is the $b_i$-derivative of
$
\theta_z\bigl(de(\partial_u)\bigr)=\ell,
$
which is zero by our normalization.  Hence
$
\partial_u A_i(b,0)=0.
$
A global vector field on $\Pp^1$ has a coefficient of degree at most two in the affine coordinate $u$.  Thus
$
A_i(b,u)=u^2a_i(b)
$
for some holomorphic functions $a_i$ on $B$.  Put
$
\alpha=\sum_i a_i(b)\,db_i.
$
In these coordinates the contact distribution is
\[
D=\ker\lambda,\qquad \lambda=du+u^2\alpha,\qquad V=u^2\partial_u.
\]
To check that $V$ preserves this distribution, compute
\[
\mathrm d\lambda=2u\,\mathrm du\wedge\alpha+u^2\,\mathrm d\alpha,
\qquad
\lambda(V)=u^2,
\]
and hence
\[
\mathcal L_V\lambda
=\iota_V\mathrm d\lambda+\mathrm d(\lambda(V))
=2u\lambda.
\]
If $W$ is a vector field in $D$, then $\lambda(W)=0$, and the preceding identity gives
\[
\lambda([V,W])=-(\mathcal L_V\lambda)(W)=0.
\]
Thus $[V,W]$ is again in $D$, so $V$ preserves the contact distribution.

At $u=\infty$, use the second affine chart.  In the coordinate $v=-1/u$,
$
\lambda=v^{-2}(\mathrm dv+\alpha),$ $ V=\partial_v.$
Multiplying a contact form by a nonzero function does not change its kernel, so near $u=\infty$ we have $D=\ker(\mathrm dv+\alpha)$, which is preserved by this vector field.  Along the chosen curve, the contact Hamiltonian is $\lambda(V)=u^2$ in the $u$-chart and is $1$ in the normalized $v$-chart at infinity.  Hence it is nonzero at every point other than $z$.
\end{proof}

\section{First jets and minimal rational tangents}\label{sec:jets}

In this section, let $Y$ be a smooth connected projective variety. We write
$\operatorname{RatCurves}^n(Y)$ for the normalized space of rational curves
on $Y$; the superscript $n$ denotes normalization. An irreducible component
$K$ is called dominating if the union of its general curves is dense in $Y$.
A parametrized curve $\nu:\Pp^1\to Y$ is called \emph{free} when the
bundle $\nu^*T_Y$ is globally generated.  By the
Birkhoff--Grothendieck theorem, this is equivalent to the condition that
all summands in the splitting
\[
\nu^*T_Y\simeq\bigoplus_i\mathcal O_{\Pp^1}(a_i)
\]
have $a_i\geq0$.  In particular, a free curve satisfies
$H^1(\Pp^1,\nu^*T_Y)=0$.  The vanishing of this cohomology group is a
weaker condition, since $H^1(\Pp^1,\mathcal O(a))=0$ already for
$a\geq-1$.  Assume that the general member of $K$ is free, and let
$\mathcal H$ be the open locus of free birational parametrizations of
curves in $K$ inside the corresponding irreducible component of
$\operatorname{Hom}(\Pp^1,Y)$.

On the open set where $\nu$ is immersive, the tangent map is
\[
\tau:\mathcal H\times\Pp^1\dashrightarrow\Pp(T_Y),\qquad
(\nu,t)\longmapsto [\dd\nu_t(T_t\Pp^1)].
\]
Let $\mathcal C$ be the reduced closure of the image of this map, and write
\[
p:\mathcal C\longrightarrow Y
\]
for the projection induced by $\Pp(T_Y)\to Y$. The fibre $\mathcal C_y$
consists of the limiting tangent directions at $y$ of curves from $K$. A
local map preserves $\mathcal C$ if its projectivized differential preserves
the radical ideal sheaf $\mathcal I_{\mathcal C}^{\mathrm{red}}$.

\begin{lemma}\label{lem:jets}
Let $E$ be a globally generated vector bundle on $\Pp^1$.  Let $g:B\to\Pp^1$ be finite, with $B$ smooth and connected, and let $b\in B$ be unramified with $t=g(b)$.  Then every $s\in H^0(B,g^*E)$ has a section $s_0\in H^0(\Pp^1,E)$ whose first jet at $t$ equals the first jet of $s$ at $b$, transported by the local inverse of $g$.
\end{lemma}
\begin{proof}
By the Birkhoff--Grothendieck theorem,
$E=\bigoplus_i\OO(a_i)$ with $a_i\ge0$. We treat each summand separately.
If $a_i=0$, then $g^*\OO$ is trivial. Every global section is constant,
since $B$ is connected and projective, and is therefore the pullback of a
section on $\Pp^1$.

Suppose $a_i\ge1$.  The length-two subscheme $2t\subset\Pp^1$ records the
 value and the first derivative of a section at $t$.  The exact sequence
 \[
 0\longrightarrow\OO(a_i-2)\longrightarrow\OO(a_i)
 \longrightarrow\OO(a_i)|_{2t}\longrightarrow0
 \]
has $H^1(\Pp^1,\OO(a_i-2))=0$. Restriction to $2t$ is therefore surjective,
so a section of $\OO(a_i)$ can be chosen with any prescribed value and first
derivative at $t$. Because $g$ is unramified at $b$, it has a local inverse
there, so the first jet of $s$ at $b$ can be transported to the first jet at
$t$. Choosing a matching section in every summand gives the required $s_0$.
\end{proof}

For a local vector field $v$ on $Y$, write $\widehat v$ for its projectivized differential lift to $\Pp(T_Y)$.

\begin{proposition}\label{prop:firstjet}
Let $f:R\to Y$ be surjective and generically finite, with $R$ smooth and connected projective, and let $S\in H^0(R,f^*T_Y)$.  Suppose $f$ is biholomorphic from an analytic open set $U\subset R$ onto $V\subset Y$.  Transport $S$ to a vector field $v$ on $V$.  Then $\widehat v$ preserves the radical ideal sheaf of the reduced analytic space $\mathcal C|_V\subset\Pp(T_Y)|_V$.
\end{proposition}
\begin{proof}
Choose a general member $\nu:\Pp^1\to Y$ of the free component $\mathcal H$
and a general point $t_0$ at which $\nu$ is immersive, with
$y=\nu(t_0)\in V$.  We may also assume that the corresponding point of the
tangent-image closure is smooth.  Put $r=(f|_U)^{-1}(y)$.  Take the
irreducible component of $R\times_Y\Pp^1$ through $(r,t_0)$ which dominates
$\Pp^1$, and let $B$ be its normalization.  Since the non-finite locus of
$f$ is a proper closed subset and $K$ is dominating, a general member $\nu$
is not contained in that locus.  The chosen component is a proper curve and
its projection to $\Pp^1$ is therefore generically finite, because a general
fibre is contained in a finite fibre of $f$.  Hence
this projection is finite.  The other projection is finite because it is induced
by the base change of the finite map $\nu:\Pp^1\to Y$.  Thus the projections
give finite maps
\[
g:B\longrightarrow\Pp^1,\qquad j:B\longrightarrow R,
\qquad f\circ j=\nu\circ g.
\]
Because $f|_U$ is biholomorphic and $r\in U$, the map $g$ is unramified at
the point above $(r,t_0)$.  Set $E=\nu^*T_Y$.  Pulling back $S$ gives a
section $j^*S\in H^0(B,g^*E)$, and Lemma~\ref{lem:jets} supplies a section
$s_0\in H^0(\Pp^1,E)$ with the same value and first derivative at $t_0$.

Since $\nu$ is free, $H^1(\Pp^1,E)=0$.  Thus the Hom-scheme is smooth at
$\nu$, with tangent space $H^0(\Pp^1,E)$; because $\mathcal H$ is open in
the corresponding component, $s_0$ is tangent to a holomorphic one-parameter
family $\nu_\tau$ in $\mathcal H$ through $\nu$.  In local coordinates on
$Y$ and on $\Pp^1$, the value and first derivative of $s_0$ are exactly the
first-order variation of $\nu_\tau(t_0)$ and of
$\dd\nu_\tau|_{t_0}$.  The identity $f\circ j=\nu\circ g$, together with
the local inverse of $g$, identifies the first jet of $j^*S$ with the first
jet of the section $v\circ\nu$ at $t_0$; hence $s_0$ has this same first jet
and its value is $v(y)$.  Consequently, the tangent vector of the
projectivized differential of $v$ at
$[\dd\nu_{t_0}(T_{t_0}\Pp^1)]$ is tangent to the image of the tangent map.
Varying the general pair $(\nu,t_0)$ shows that $\widehat v$ is tangent to
$\mathcal C$ at a dense set of smooth points over $V$.

If $h$ is a local holomorphic function vanishing on the reduced space
$\mathcal C|_V$, then $\widehat v(h)$ vanishes on this dense smooth locus.
The identity principle therefore gives $\widehat v(h)|_{\mathcal C}=0$,
which is precisely
$\widehat v(\mathcal I_{\mathcal C}^{\mathrm{red}})\subset
\mathcal I_{\mathcal C}^{\mathrm{red}}$.  Hence the radical ideal is
preserved on all of $\mathcal C|_V$.
\end{proof}

\begin{corollary}\label{cor:family}
Let $F:R\times\Delta\to Y$ be holomorphic, with every $f_\tau=F|_{R\times\{\tau\}}$ surjective and generically finite.  Near a point where $f_0$ is \'etale, the local biholomorphisms $\varphi_\tau=f_\tau\circ f_0^{-1}$ preserve $\mathcal C$ under their projectivized differentials.
\end{corollary}
\begin{proof}
Differentiating $F$ in the parameter $\tau$ gives, for each fixed $\tau$,
a section of $f_\tau^*T_Y$. Where $f_0$ is \'etale, the maps
$\varphi_\tau=f_\tau\circ f_0^{-1}$ are local biholomorphisms and form the
local flow of a time-dependent vector field. By Proposition~\ref{prop:firstjet},
its projectivized lift preserves the radical ideal of $\mathcal C$.

Put $\Psi_\tau=\Pp(\mathrm d\varphi_\tau)$ for the projectivized
differential.  Let $h_1,\ldots,h_m$ generate this ideal locally on
$\Pp(T_Y)$. Their pullbacks along $\Psi_\tau$ satisfy the homogeneous
linear system
\[
\frac{\dd}{\dd\tau}(\Psi_\tau^*h_j)
  =\sum_k \bigl(a_{jk}(\tau,\mathord\cdot)\circ\Psi_\tau\bigr)
    \,\Psi_\tau^*h_k.
\]
All $h_j$ vanish on $\mathcal C$ at the initial parameter, so uniqueness for
this system gives zero for all parameters. Thus
$\Psi_\tau(\mathcal C)\subset\mathcal C$; applying the same argument to
the inverse flow gives the reverse inclusion. The projectivized
differentials of $\varphi_\tau$ therefore preserve $\mathcal C$.
\end{proof}

\section{The family of minimal rational tangents}\label{sec:vmrt}

Assume $\Pic(Y)=\mathbb Z[L]$.  Choose an irreducible component $K$ of the
family of contact lines, namely rational curves $C$ with $\deg_L C=1$, whose
universal family dominates $Y$; such a component exists because contact lines
pass through every point by \cite[Theorem~3.1(i)]{Buczynski10}.  A general
contact line is free by \cite[Lemma~3.5]{Kebekus01}.  For $y\in Y$, the
\emph{variety of minimal rational tangents} (VMRT) is the closure in
$\Pp(T_yY)$ of the tangent lines at $y$ to curves from $K$ passing through
$y$.  We use the deformation theory of free rational curves from
\cite{Kebekus01} and the corrected contact-line theorem of
\cite[Theorem~3.1(iii) and Remark~3.2]{Buczynski10}.

\begin{theorem}[{\cite[Theorem~3.1(iii) and Remark~3.2]{Buczynski10}}]\label{thm:contact-lines}
For a general $y\in Y$, the reduced variety of minimal rational tangents associated with $K$ is smooth of dimension $n-1$ and Legendrian in $\Pp(D_y)$.  Every contact line from the chosen family through $y$ is smooth, and two distinct such lines through $y$ have distinct tangent directions at $y$.
\end{theorem}

Here $D_y\subset T_yY$ is the contact hyperplane. A subvariety of
$\Pp(D_y)$ is Legendrian if it is tangent to the induced contact structure
and has the maximal dimension allowed by that condition.

Let $\mathcal C\subset\Pp(T_Y)$ be the integral reduced tangent-image closure for this fixed component.

\begin{lemma}\label{lem:generic-smooth}
The morphism $p:\mathcal C\to Y$ is smooth over a nonempty Zariski-open subset of $Y$.
\end{lemma}
\begin{proof}
A curve of $L$-degree one cannot degenerate into two nonconstant
components, since each would have positive integral $L$-degree.  The same stability argument used in
Lemma~\ref{lem:fibre} rules out contracted components.  Thus the compactified
family and its universal curve are proper.  Over the open set supplied by
Theorem~\ref{thm:contact-lines}, the curves are smooth and immersed at the
marked point, so the tangent map is a morphism and its image in each fibre
is closed.

Theorem~\ref{thm:contact-lines} says that, for general $y$, every contact
line through $y$ is smooth, the tangent directions are distinct, and the
reduced VMRT is smooth.  Over this open set the tangent map from the
normalised family is generically one-to-one and unramified at the general
point; in characteristic zero its image has the reduced generic fibre
described by the VMRT.  Thus the generic fibre of $p$ is geometrically
smooth.  Generic flatness gives a nonempty open subset over which $p$ is
flat, and smoothness of the fibres is open in a flat morphism of finite
presentation.  After shrinking this subset, all fibres are smooth there.
Hence $p$ is smooth over a nonempty Zariski-open subset of $Y$.
\end{proof}

\begin{lemma}\label{lem:transitivity}
Let $Y$ be smooth, connected and projective, and let $\mathcal C\subset\Pp(T_Y)$ be integral and dominating.  Suppose $p$ is smooth over a nonempty open set.  If at every $y\in Y$ there are local one-parameter families of biholomorphisms preserving $\mathcal C$ whose infinitesimal values span $T_yY$, then $p$ is smooth everywhere.  If $Y$ is simply connected, every fibre is connected and irreducible.
\end{lemma}
\begin{proof}
Let $N\subset\mathcal C$ be the nonsmooth locus of $p$, and let
$A=\overline{p(N)}$.  Since $p$ is smooth over a nonempty open subset,
$A$ is a proper closed subset of $Y$ (the map $p$ is proper because
$\mathcal C$ is closed in the projective bundle $\Pp(T_Y)$).  A local
biholomorphism preserving $\mathcal C$ also preserves the property of being a nonsmooth point of
$p$, and therefore preserves $A$.

Fix $y\in A$. Compose $\dim Y$ of the given local families whose
infinitesimal values form a basis at $y$. The resulting orbit map has
invertible differential at the origin. Its image therefore contains a
neighbourhood of $y$, which must lie in $A$. This contradicts the fact that
$A$ is a proper closed algebraic subset of the connected smooth variety
$Y$. Thus $A$ is empty, and $p$ is smooth everywhere.

The Stein factorization of $p$ is
\[
\mathcal C\longrightarrow Y'\longrightarrow Y,
\]
The second map is finite \'etale, since $p$ is proper and smooth.  The irreducibility of $\mathcal C$ makes $Y'$ connected.  If
$Y$ is simply connected, it has no nontrivial connected finite \'etale
cover, so $Y'=Y$.  Hence the fibres of $p$ are connected.  A smooth
connected fibre is irreducible, because its irreducible components are
open and disjoint.
\end{proof}

\begin{lemma}\label{lem:span}
For every $w\in Y$ there are local biholomorphism families preserving $\mathcal C$ whose infinitesimal values span $T_wY$.
\end{lemma}
\begin{proof}
Let $C_0$ be the twistor fibre through $w$, regarded as a rational curve in
$Y$.  Its normal bundle is
\[
N_{C_0/Y}\simeq\OO(1)^{\oplus 2n}.
\]
Consider deformations of $C_0$ that continue to pass through $w$.  The
differential of the map sending such a deformation to its tangent direction
at $w$ is
\[
H^0(C_0,N_{C_0/Y}(-w))
 \longrightarrow (N_{C_0/Y})_w\otimes T_w^*C_0.
\]
The kernel and cokernel are respectively
$H^0(C_0,N_{C_0/Y}(-2w))$ and
$H^1(C_0,N_{C_0/Y}(-2w))$.  Since
$N_{C_0/Y}(-2w)=\OO(-1)^{\oplus2n}$, both groups vanish.  Thus the
tangent-direction map is locally biholomorphic at $C_0$.

Consequently, nearby curves through $w$ produce an analytic open subset of
$\Pp(T_wY)$.  We can choose $2n+1$ of their tangent lines whose
representatives span $T_wY$.  For each selected curve $C_i$, choose a point
$z_i\in C_i\setminus\{w\}$ as the fixed first endpoint, and use $w$ as the
second endpoint in the ordered incidence construction of
Sections~\ref{sec:action}--\ref{sec:contactfield}.  This gives a projective
source $W_i$, an evaluation map $E_i:W_i\to Y$, and an algebraic additive
action $g_{i,\tau}$ on $W_i$.  At the \'etale point $(C_i,z_i,w)$,
Proposition~\ref{prop:contact} gives a nonzero vector tangent to $C_i$ at $w$.
The local motions are thus constructed from deformations of twistor
fibres of degree two, not from the degree-one curves in $K$.

Take a projective resolution $r_i:R_i\to W_i$ that is an isomorphism near
the chosen \'etale point.  The maps
\[
f_{i,\tau}=E_i\circ g_{i,\tau}\circ r_i:R_i\longrightarrow Y
\]
are surjective and generically finite.  Corollary~\ref{cor:family} gives
local biholomorphisms preserving $\mathcal C$.  Their infinitesimal values
at $w$ are nonzero multiples of the selected tangent directions, and these
directions span $T_wY$.
\end{proof}

\begin{lemma}\label{lem:simply}
The twistor space $Y$ is simply connected.
\end{lemma}
\begin{proof}
The K\"ahler--Einstein metric on $Y$ has positive Ricci curvature
\cite[Theorem~2.1]{LeBrun95}.  Myers'
theorem \cite{Myers41} therefore shows that $\pi_1(Y)$ is finite.  Let
$q:\widetilde Y\to Y$ be the universal finite unramified cover, of degree
$d$.  The pullback metric still has positive Ricci curvature, and
$K_{\widetilde Y}^{-1}=q^*K_Y^{-1}$ is ample.  Kodaira vanishing
\cite{Kodaira53} gives
\[
H^j(Y,\OO_Y)=H^j(\widetilde Y,\OO_{\widetilde Y})=0
\qquad (j>0).
\]
Both holomorphic Euler characteristics are therefore one.
Hirzebruch--Riemann--Roch \cite{Hirzebruch54}, applied to the finite
unramified cover, gives
\[
1=\chi(\widetilde Y,\OO_{\widetilde Y})
 =d\,\chi(Y,\OO_Y)=d.
\]
Thus $d=1$, so the universal cover is trivial and $Y$ is simply connected.
\end{proof}

\begin{proposition}\label{prop:nonlinear}
The morphism $p:\mathcal C\to Y$ is smooth with connected irreducible fibres.  Its general fibre is nonlinear, in the sense that it is not a finite union of projective linear subspaces.
\end{proposition}
\begin{proof}
The first assertion follows from Lemmas~\ref{lem:generic-smooth},
\ref{lem:span}, \ref{lem:simply}, and \ref{lem:transitivity}.

Suppose, for contradiction, that a general fibre of $p$ is a finite union
of projective linear subspaces.  The fibre is irreducible, so it must be a
single linear space $\Pp^{n-1}$.  Because the VMRT lies in
$\Pp(D_y)$, its linear span defines, on a dense open subset of $Y$, a
rank-$n$ holomorphic distribution $\mathcal D\subset D\subset T_Y$.
The linear-VMRT integrability theorem
\cite[Proposition~7.4]{Hwang11} says that, after passing to the
\'etale locus of the relevant local finite cover, this distribution is
integrable; see also \cite[Theorem~3.1]{Araujo06} and
\cite[Proposition~1]{Hwang07}.  Integrability means that local sections of
$\mathcal D$ are closed under Lie brackets, so $\mathcal D$ is tangent to
local leaves of dimension $n$.

On the other hand, the distribution generated by the minimal rational
tangent directions is bracket-generating by
\cite[Proposition~6.8]{Hwang11}; see also \cite[Proposition~13]{HwangMok98}.
Thus local sections of $\mathcal D$ and their iterated Lie brackets
span the tangent space of dimension $2n+1$.
The bracket-generating property is unchanged after the same local pullback.
For an integrable rank-$n$ distribution, all these brackets remain
tangent to its $n$-dimensional leaves. This contradiction proves the
assertion.
\end{proof}

\section{Global contact fields and the classification}\label{sec:classification}

We apply the following theorem of Hwang and Mok to obtain global
vector fields on $Y$. A Fano manifold is a smooth projective variety with
ample anticanonical bundle; Picard number one means that its
N\'eron--Severi space has dimension one. A regular family over the disk
$\Delta$ is a smooth proper family. In the theorem below, its fibres
are assumed to be Fano. A map is generically finite if it is finite over
a dense open subset of its target.

\begin{theorem}[{\cite[Theorem~10, pp.~189--190]{Mok2008}}]\label{thm:hwangmok}
Let $\pi:\mathscr Y\to\Delta$ be a regular family of Fano manifolds of Picard number one.  Assume that $Y_0=\pi^{-1}(0)$ has a minimal rational component whose general variety of minimal rational tangents is nonlinear.  Let $R$ be a projective manifold and let $F:R\times\Delta\to\mathscr Y$ be a holomorphic map over $\Delta$ such that each fibre map $f_\tau$ is surjective and generically finite.  After shrinking the parameter disk, there is a holomorphic family of biholomorphisms $\Phi_\tau:Y_0\to Y_\tau$ with $\Phi_0=\mathrm{id}_{Y_0}$ and $f_\tau=\Phi_\tau\circ f_0$.
\end{theorem}

The theorem identifies the fibres by biholomorphisms under which the
given family of maps is constant.

A vector field is contact if its local flow preserves $D$. Its
contact Hamiltonian is the section $h=\theta(H)$ of $L$.

\begin{proposition}\label{prop:globalfield}
Fix $z\in Y$ and $\ell\in L_z\setminus\{0\}$.  Let $W_z$ be the
projective normalized incidence source constructed above, chosen to contain
the real twistor fibre through $z$.  Then there is a global holomorphic
contact vector field $H_z$ on $Y$ which agrees with $(E_z)_*V$ on every
\'etale chart coming from the balanced transverse incidence locus.  Its
Hamiltonian $h_z=\theta(H_z)\in H^0(Y,L)$ satisfies
\begin{equation}\label{eq:nonzero-hamiltonian}
h_z(\sigma z)\ne0.
\end{equation}
\end{proposition}
\begin{proof}
Choose a projective resolution $r:R\to W_z$ that is an isomorphism near
the smooth incidence locus.  Put $\mathscr Y=Y\times\Delta$ and define
\[
F(q,\tau)=(E_z(g_{z,\tau}(r(q))),\tau).
\]
Write $f_\tau(q)=E_z(g_{z,\tau}(r(q)))$.  Lemma~\ref{lem:etale}
shows that $f_0$ is generically finite and surjective.  Since each
$g_{z,\tau}$ is an automorphism, the same is true for every $f_\tau$.
Proposition~\ref{prop:nonlinear} verifies the VMRT hypothesis of
Theorem~\ref{thm:hwangmok}.  Applied to the constant family
$\mathscr Y=Y\times\Delta$, that theorem gives biholomorphisms
$\Phi_\tau:Y\to Y$ with
\[
f_\tau=\Phi_\tau\circ f_0,\qquad \Phi_0=\mathrm{id}_Y.
\]
Differentiation at $\tau=0$ gives a global vector field $H_z$,
which agrees with $(E_z)_*V$ on every \'etale incidence chart.

By Proposition~\ref{prop:contact}, $H_z$ preserves $D$ on a nonempty
open set.  For a local section $\xi$ of $D$, the quantity
$\theta([H_z,\xi])$ depends $\OO_Y$-linearly on $\xi$, and therefore defines
a section of $\operatorname{Hom}(D,L)$.  This section is zero on the open
set just considered, hence is zero everywhere.  Thus $H_z$ is a global
contact vector field and $h_z=\theta(H_z)$ is a global contact Hamiltonian.

The involution $\sigma$ preserves the real twistor fibre through
$z$, and $\sigma z\ne z$. At $\sigma z$, the field $(E_z)_*V$ is
nonzero and tangent to this transverse curve.
Lemma~\ref{lem:etale} and Proposition~\ref{prop:contact} therefore give
$h_z(\sigma z)\ne0$.
\end{proof}

\begin{corollary}\label{cor:globally-generated}
The contact line bundle $L$ is globally generated.
\end{corollary}
\begin{proof}
Fix $w\in Y$ and apply Proposition~\ref{prop:globalfield} with
$z=\sigma w$.  Since $\sigma^2=\mathrm{id}$, we have $\sigma z=\sigma^2w=w$.
Equation~\eqref{eq:nonzero-hamiltonian} gives a section $h_z$ of $L$
nonzero at $w$. As $w$ was arbitrary, the evaluation map
\[
H^0(Y,L)\otimes\OO_Y\longrightarrow L
\]
is surjective. Hence $L$ is globally generated.
\end{proof}

\begin{proof}[Proof of Theorem~\ref{thm:main}]
The first two cases in \eqref{eq:intro-reduction} are the standard twistor
spaces already identified by LeBrun.  Assume therefore that
$\Pic(Y)=\mathbb Z[L]$.  By Corollary~\ref{cor:globally-generated}, the
complete linear system of $L$ defines a morphism
\[
\varphi_L:Y\longrightarrow\Pp(H^0(Y,L)^*),\qquad L=\varphi_L^*\OO(1).
\]
This morphism is finite onto its image.  Indeed, a positive-dimensional fibre
would contain a curve $B$ on which $L$ has degree zero, contradicting the
ampleness of $L$.

The morphism $\varphi_L$ satisfies condition (H1) of Beauville's
recognition theorem. Condition (H2) requires the connected contact
automorphism group to be reductive. For positive quaternionic--K\"ahler
twistor spaces, this follows from LeBrun's result \cite{LeBrun95}, as
used in \cite[Theorem~0.1]{Beauville98}.

Beauville's theorem identifies $(Y,D)$ with the projectivized minimal
nonzero nilpotent orbit of a simple complex Lie algebra, which is the twistor
space of the corresponding Wolf space.  LeBrun's uniqueness theorem
\cite[Corollary~3.4]{LeBrun95} then identifies the positive
quaternionic--K\"ahler metric after scalar-curvature normalization.  Thus $g$
is the standard Wolf-space metric up to a positive constant factor.
\end{proof}

\noindent\textbf{Acknowledgments.}
The first named author is supported by Overseas High-Level Young Talent Recruitment Programs and is partially supported by the Chongqing Natural Science Foundation Innovation and Development Joint Fund CSTB2023NSCQ-LZX0031.
The second named author is partially supported by the National Key Research and Development Program of China (Grant No.~2024YFA1013201), the Natural Science Foundation of Chongqing (Grant No.~CSTB2024NSCQ-LZX0040), and the Special Project of Chongqing Municipal Science and Technology Bureau (Grant No.~2025CCZ015). The third named author is sponsored by the National Key R\&D Program of China (Grant No.~2024YFA1013200) and the National Natural Science Foundation of China (Grant No.~12671100).

\noindent\textbf{AI Declaration}: The authors used ChatGPT as an auxiliary tool in this work. The authors verified and completed all mathematical arguments and take full responsibility for the content of this paper.

\end{document}